\documentclass{amsart}
\usepackage{setspace}
\usepackage{a4}
\usepackage{amsthm}
\usepackage{latexsym}
\usepackage{amsfonts}
\usepackage{graphicx}
\usepackage{textcomp}
\usepackage{cite}
\usepackage{enumerate}
\usepackage{amssymb}
\usepackage{hyperref}
\usepackage{amsmath}
\usepackage{tikz}
\usepackage[mathscr]{euscript}
\usepackage{mathtools}
\newtheorem{theorem}{Theorem}[section]

\newtheorem{definition}[theorem]{Definition}

\newtheorem{problem}[theorem]{Problem}

\title{This is the title}
\usepackage{fancyhdr}

\begin{document}
\hrule\hrule\hrule\hrule\hrule
\vspace{0.3cm}	
\begin{center}
{\bf\large{{Noncommutative Donoho-Elad-Gribonval-Nielsen-Fuchs Sparsity Theorem}}}\\
\vspace{0.3cm}
\hrule\hrule\hrule\hrule\hrule
\vspace{0.3cm}
\textbf{K. Mahesh Krishna}\\
School of Mathematics and Natural Sciences\\
Chanakya University Global Campus\\
NH-648, Haraluru Village\\
Devanahalli Taluk, 	Bengaluru  Rural District\\
Karnataka State 562 110 India\\
Email: kmaheshak@gmail.com\\

Date: \today
\end{center}

\hrule\hrule
\vspace{0.5cm}
%--------------------------------------
\textbf{Abstract}: Breakthrough Sparsity Theorem, derived  independently by Donoho and Elad   \textit{[Proc. Natl. Acad. Sci. USA, 2003]},  Gribonval and Nielsen \textit{[IEEE Trans. Inform. Theory, 2003]} and Fuchs \textit{[IEEE Trans. Inform. Theory, 2004]}  says that unique  sparse solution to NP-Hard $\ell_0$-minimization problem can be obtained using unique solution of P-Type $\ell_1$-minimization problem. In this paper, we derive noncommutative version of their result using frames for Hilbert C*-modules.



\hrule

%\tableofcontents
\hrule
\section{Introduction}

Let $\mathcal{H}$ be a  finite dimensional Hilbert  space over $\mathbb{K}$ ($\mathbb{C}$ or $\mathbb{R}$). A  finite collection $\{\tau_j\}_{j=1}^n$ in $\mathcal{H}$ is said to be a \textbf{frame} (also known as \textbf{dictionary}) \cite{BEDETTOFICKUS, HANKORNELSON} for $\mathcal{H}$ if there are $a, b >0$ such that 
\begin{align*}
	a\|h\|^2\leq \sum_{j=1}^{n}|\langle h,  \tau_j \rangle|^2\leq b \|h\|^2, \quad \forall h \in \mathcal{H}.
\end{align*}
 A frame $\{\tau_j\}_{j=1}^n$ for  $\mathcal{H}$ is said to be \textbf{normalized} if $\|\tau_j\|=1$ for all $1\leq j \leq n$. Note that any frame can be normalized by dividing each element by its norm. Given a frame $\{\tau_j\}_{j=1}^n$ for  $\mathcal{H}$, we define the analysis operator
\begin{align*}
	\theta_\tau:\mathcal{H}\ni h \mapsto \theta_\tau h\coloneqq (\langle h,  \tau_j \rangle)_{j=1}^n \in \mathbb{K}^n.
\end{align*}
Adjoint of the analysis operator is known as  the synthesis operator whose expression  is 
\begin{align*}
\theta_\tau^*: \mathbb{K}^n \ni (a_j)_{j=1}^n \mapsto \theta_\tau^*(a_j)_{j=1}^n\coloneqq\sum_{j=1}^{n}a_j\tau_j \in \mathcal{H}.
\end{align*}
Given $d\in \mathbb{K}^n$, let $\|d\|_0$ be the number of nonzero entries in $d$.   Following $\ell_0$-minimization problem appears in many  of electronic devices.
\begin{problem} \label{P0}
 Let $\{\tau_j\}_{j=1}^n$ be a normalized frame for  $\mathcal{H}$.  Given $h \in \mathcal{H}$, solve 	
\begin{align*}
	\text{minimize }\{\|d\|_0:d\in \mathbb{K}^n\} \quad \text{ subject to } \quad \theta_\tau^*d=h.
\end{align*}
\end{problem}
Recall that $c \in \mathbb{K}^n$ is said to be a unique solution to Problem \ref{P0}  if it satisfies following two conditions.
\begin{enumerate}[\upshape(i)]
	\item $\theta_\tau^*c=h$.
	\item If $d \in \mathbb{K}^n$ satisfies $\theta_\tau^*d=h$, then 
\begin{align*}
	\|d\|_0>\|c\|_0.
\end{align*}
\end{enumerate}
In 1995, Natarajan showed that Problem \ref{P0} is NP-Hard \cite{NATARAJAN, FOUCARTRAUHUT}. 
As the operator $\theta_\tau^*$ is surjective, for a given $h \in \mathcal{H}$, there is a $d\in \mathbb{K}^n$ such that  $\theta_\tau^*d=h.$ Thus the central problem is to say when the  solution to Problem \ref{P0} is unique.   It is well-known that  \cite{CHENDONOHOSAUNDERS, DONOHOHUO, BRUCKSTEINDONOHOELAD}   following problem is the closest convex relaxation problem to Problem \ref{P0}.
\begin{problem} \label{P1}
Let $\{\tau_j\}_{j=1}^n$ be a normalized frame for  $\mathcal{H}$.	Given $h \in \mathcal{H}$, solve 	
	\begin{align*}
		\text{minimize }\{\|d\|_1:d\in \mathbb{K}^n\} \quad \text{ subject to } \quad \theta_\tau^*d=h.
	\end{align*}
\end{problem}
There are several linear programmings available to obtain solution of Problem \ref{P1} and it  is a P-problem \cite{XUEYE, TERLAKY, TILLMANN}.\\
Most important result which shows that by solving Problem \ref{P1} we also get a solution to Problem \ref{P0} is obtained independently by Donoho and Elad \cite{DONOHOELAD},  Gribonval and Nielsen \cite{GRIBONVALNIELSEN} and Fuchs \cite{FUCHS, FUCHS2} is the following.
\begin{theorem}\cite{DONOHOELAD, GRIBONVALNIELSEN, KUTYNIOK, ELAD, FUCHS2, FUCHS}\label{DEGN} (\textbf{Donoho-Elad-Gribonval-Nielsen-Fuchs Sparsity Theorem})
	Let $\{\tau_j\}_{j=1}^n$ be a  normalized frame  for  $\mathcal{H}$. If  
 $h \in \mathcal{H}$ can be written as  $	h=\theta_\tau^*c$ for some  $c\in \mathbb{K}^n$ satisfying 
 \begin{align*}
 \|c\|_0<\frac{1}{2}\left(1+\frac{1}{\displaystyle\max_{1\leq j, k \leq n, j\neq k}|\langle \tau_j, \tau_k\rangle|}\right),
 \end{align*}
 then $c$ is the unique solution to Problem \ref{P1} and Problem \ref{P0}.
\end{theorem}
Our fundamental motivation comes from the following question: What is the noncommutative analogue of Theorem \ref{DEGN}? This is then naturally connected with the notion of Hilbert C*-modules which  are first introduced by Kaplansky \cite{KAPLANSKY} for modules over commutative C*-algebras and later developed for modules over arbitrary C*-algebras by Paschke  \cite{PASCHKE} and Rieffel \cite{RIEFFEL}. We end the introduction by recalling the definition of Hilbert C*-modules. 
\begin{definition}\cite{KAPLANSKY, PASCHKE, RIEFFEL}
	Let $\mathcal{A}$ be a  unital C*-algebra. A left module 	 $\mathcal{E}$  over $\mathcal{A}$ is said to be a (left) Hilbert C*-module if there exists a  map $ \langle \cdot, \cdot \rangle: \mathcal{E}\times \mathcal{E} \to \mathcal{A}$ such that the following hold. 
	\begin{enumerate}[\upshape(i)]
		\item $\langle x, x \rangle  \geq 0$, $\forall x \in \mathcal{E}$. If $x \in  \mathcal{E}$ satisfies $\langle x, x \rangle  =0 $, then $x=0$.
		\item $\langle x+y, z \rangle  =\langle x, z \rangle+\langle y, z \rangle$, $\forall x,y,z \in \mathcal{E}$.
		\item  $\langle ax, y \rangle  =a\langle x, y \rangle$, $\forall x,y \in \mathcal{E}$, $\forall a \in \mathcal{A}$.
		\item $\langle x, y \rangle=\langle y,x \rangle^*$, $\forall x,y \in \mathcal{E}$.
		\item $\mathcal{E}$ is complete w.r.t. the norm $\|x\|\coloneqq \sqrt{\|\langle x, x \rangle\|}$,  $\forall x \in \mathcal{E}$.
	\end{enumerate}
\end{definition}

\section{Noncommutative Donoho-Elad-Gribonval-Nielsen-Fuchs Sparsity Theorem}
 Observe that the notion of frames is needed for Theorem \ref{DEGN}. Thus we want noncommutative frames. These are introduced in 2002 by Frank and Larson in their seminal paper \cite{FRANKLARSON}. We begin by  recalling  the definition of noncommutative  frames for Hilbert C*-modules. This notion is already well-developed in parallel with Hilbert space frame theory \cite{RAEBURNTHOMPSON, ARAMBASIC, HANJINGMOHAPATRA}. In the paper, 	we consider only finite rank modules. 
\begin{definition}\cite{FRANKLARSON}
	Let 	 $\mathcal{E}$ be a Hilbert C*-module over a  unital C*-algebra $\mathcal{A}$. A collection $\{\tau_j\}_{j=1}^n $ in  $\mathcal{E}$ is said to be a (modular) \textbf{frame} for  $\mathcal{E}$ if there are real $a, b>0$ such that 
	\begin{align*}
	a	\langle x, x \rangle \leq   \sum_{j=1}^n  \langle x, \tau_j \rangle \langle \tau_j, x \rangle \leq b \langle x, x \rangle,   \quad \forall x \in \mathcal{E}.
	\end{align*}
\end{definition}
A collection  $\{\tau_j\}_{j=1}^n $ in a Hilbert C*-module $\mathcal{E}$  over unital C*-algebra $\mathcal{A}$ with identity $1$  is said to have \textbf{unit inner product} if
\begin{align*}
	\langle \tau_j, \tau_j \rangle =1, \quad \forall 1\leq j \leq n.
\end{align*}
Let  $\mathcal{A}$ be a unital C*-algebra. For $n \in \mathbb{N}$, let $\mathcal{A}^n$ be the standard left Hilbert C*-module over $\mathcal{A}$ with inner product 
\begin{align*}
	\langle (a_j)_{j=1}^n, (b_j)_{j=1}^n\rangle \coloneqq \sum_{j=1}^{n}a_jb_j^*, \quad \forall (a_j)_{j=1}^n, (b_j)_{j=1}^n \in \mathcal{A}^n.
\end{align*}
Hence norm on $\mathcal{A}^n$ is 
\begin{align*}
	\|(a_j)_{j=1}^n\|_2\coloneqq \left\|\sum_{j=1}^{n}a_ja_j^*\right\|^\frac{1}{2}, \quad \forall (a_j)_{j=1}^n \in \mathcal{A}^n.
\end{align*}
We define 
\begin{align*}
	\|(a_j)_{j=1}^n\|_1\coloneqq \sum_{j=1}^{n}\|a_j\|, \quad \forall (a_j)_{j=1}^n \in \mathcal{A}^n.	
\end{align*}
A frame $\{\tau_j\}_{j=1}^n$ for  $\mathcal{E}$ gives the modular analysis morphism 
\begin{align*}
	\theta_\tau:\mathcal{E}\ni x \mapsto 	\theta_\tau x \coloneqq (\langle x,  \tau_j \rangle)_{j=1}^n \in \mathcal{A}^n 
\end{align*}
and the modular synthesis morphism  
\begin{align*}
	\theta_\tau^*: \mathcal{A}^n \ni (a_j)_{j=1}^n \mapsto\theta_\tau^*(a_j)_{j=1}^n\coloneqq \sum_{j=1}^{n}a_j\tau_j \in \mathcal{E}.
\end{align*}
With these notions, we generalize  Problems \ref{P0} and  \ref{P1}. In the entire paper,  $\mathcal{E}$ denotes a finite rank  Hilbert C*-module over a  unital C*-algebra $\mathcal{A}$.
\begin{problem}  \label{BP0}
Let $\{\tau_j\}_{j=1}^n $  be a unit inner product frame  for  $\mathcal{E}$.	Given $x \in \mathcal{E}$, solve 	
	\begin{align*}
		\text{minimize }\{\|d\|_0:d\in \mathcal{A}^n\} \quad \text{ subject to } \quad \theta_\tau^* d=x.
	\end{align*}
\end{problem}
\begin{problem}  \label{BP1}
	Let $\{\tau_j\}_{j=1}^n $  be a unit inner product frame for  $\mathcal{E}$.	Given $x \in \mathcal{E}$, solve 	
	\begin{align*}
		\text{minimize }\{\|d\|_1:d\in  \mathcal{A}^n\} \quad \text{ subject to } \quad \theta_\tau^* d=x.
	\end{align*}
\end{problem}
A very powerful property used to show Theorem \ref{DEGN}  is the notion of null space property (see \cite{KUTYNIOK, COHENDAHMENDEVORE}). We now define the same property for Hilbert C*-modules. We use following notations. Let $\{e_j \}_{j=1}^n$ be the canonical  basis for $\mathcal{A}^n$. Given $M\subseteq \{1, \dots, n\}$ and $d=(d_j)_{j=1}^n \in \mathcal{A}^n$,  define 
\begin{align*}
	d_M\coloneqq \sum_{j\in M}d_je_j.
\end{align*}
\begin{definition}
A unit inner product   frame  $\{\tau_j \}_{j=1}^n$ for   $\mathcal{E}$ is said to have the (modular) \textbf{null space property} (we write NSP) of order $k\in  \{1, \dots, n\}$ if 	for every $ M \subseteq  \{1, \dots, n\} $ with $o(M)\leq k$, we have
\begin{align*}
	\|d_M\|_1<\frac{1}{2}\|d\|_1,  \quad  \forall d \in \ker(\theta_\tau^*), d \neq 0.
\end{align*}
\end{definition}
We first  relate NSP with Problem \ref{BP1}.
\begin{theorem}\label{IFF}
Let  $\{\tau_j\}_{j=1}^n $  be a unit inner product frame for  $\mathcal{E}$ and let $1\leq k \leq n$. The following are equivalent.
\begin{enumerate}[\upshape(i)]
	\item  If $x \in \mathcal{E}$ can be written as  $	x=\theta_\tau^* c$ for some  $c \in \mathcal{A}^n$ satisfying $\|c\|_0\leq k$, then $c$ is the unique solution to Problem \ref{BP1}.
	\item   $\{\tau_j \}_{j=1}^n$ satisfies the  NSP of order $k$.
\end{enumerate}	
\end{theorem}
\begin{proof}
\begin{enumerate}[\upshape(i)]
	\item 	$\implies$ (ii) Let $ M \subseteq     \{1, \dots, n\} $ with $o(M)\leq k$ and let $d \in \ker(\theta_\tau^*), d \neq 0$. Then we have 
	\begin{align*}
	0=\theta_\tau^* d=\theta_\tau^*(d_M+d_{M^c})=\theta_\tau^*(d_M)+\theta_\tau^*(d_{M^c})
	\end{align*}
which gives 
\begin{align*}
	\theta_\tau^*(d_M)=\theta_\tau^*(-d_{M^c}).
\end{align*}
Define $c\coloneqq d_M \in \mathcal{A}^n$ and $x\coloneqq \theta_\tau^*(d_M)$. Then we have $\|c\|_0\leq o(M)\leq k$ and 
\begin{align*}
	x=\theta_\tau^* c=\theta_\tau^*(-d_{M^c}).
\end{align*}
By assumption (i),  we then have 
\begin{align*}
	\|c\|_1=\|d_M\|_1<\|-d_{M^c}\|_1=\|d_{M^c}\|_1.
\end{align*}
Rewriting previous inequality gives 
\begin{align*}
	\|d_M\|_1<\|d\|_1-\|d_M\|_1 \implies 	\|d_M\|_1<\frac{1}{2}\|d\|_1.
\end{align*}
Hence $\{\tau_j \}_{j=1}^n$  satisfies the  NSP of order $k$.
	\item    $\implies$ (i)  Let  $x \in \mathcal{E}$ can be written as  $	x=\theta_\tau^* c$ for some  $c \in \mathcal{A}^n$ satisfying $\|c\|_0\leq k$. Define $M\coloneqq \operatorname{supp}(c)$. Then $o(M)=\|c\|_0 \leq k$. By assumption (ii), we then have 
	\begin{align}\label{PI}
	\|d_M\|_1<\frac{1}{2}\|d\|_1,  \quad  \forall d \in \ker(\theta_\tau^*), d \neq 0.
\end{align}
Let $b\in \mathcal{A}^n$ be such that $x=\theta_\tau^* b$ and $b\neq c$. Define $a\coloneqq b-c \in \mathcal{A}^n$. Then $\theta_\tau^* a=\theta_\tau^* b-\theta_\tau^* c=x-x=0$  and hence $a \in \ker(\theta_\tau^*), a \neq 0$. Using Inequality (\ref{PI}), we get 
\begin{align}\label{P2}
	&\|a_M\|_1<\frac{1}{2}\|a\|_1 \implies \|a_M\|_1<\frac{1}{2}(\|a_{M}\|_1+\|a_{M^c}\|_1) \nonumber \\
	&\implies \|a_M\|_1<\|a_{M^c}\|_1.
\end{align}
Using Inequality (\ref{P2}) and the information that $c$ is supported on $M$, we get 
\begin{align*}
\|b\|_1-\|c\|_1&=\|b_M\|_1+\|b_{M^c}\|_1-\|c_M\|_1-\|c_{M^c}\|_1\\
&=\|b_M\|_1+\|b_{M^c}\|_1-\|c_M\|_1=\|b_M\|_1+\|(b-c)_{M^c}\|_1-\|c_M\|_1\\
&=\|b_M\|_1+\|a_{M^c}\|_1-\|c_M\|_1>\|b_M\|_1+\|a_M\|_1-\|c_M\|_1\\
&\geq \|b_M\|_1+\|(b-c)_M\|_1-\|c_M\|_1\geq \|b_M\|_1-\|b_M\|_1+\|c_M\|_1-\|c_M\|_1= 0.
\end{align*}
Hence $c$ is the unique solution to Problem \ref{BP1}.
\end{enumerate}	
\end{proof}
Using Theorem \ref{IFF} we obtain modular version of Theorem \ref{DEGN}.
\begin{theorem}\label{TDP}
Let  $\{\tau_j\}_{j=1}^n $  be a unit inner product frame for  $\mathcal{E}$.	If $x \in \mathcal{E}$ can be written as  $	x=\theta_\tau^* c$ for some  $c \in \mathcal{A}^n$ satisfying 
\begin{align}\label{I}
\|c\|_0<\frac{1}{2}\left(1+\frac{1}{\displaystyle\max_{1\leq j, k \leq n, j\neq k}\|\langle \tau_j, \tau_k\rangle\|}\right),	
\end{align}
then $c$ is the unique solution to Problem \ref{BP1}.
\end{theorem}
\begin{proof}
We show that 	$\{\tau_j\}_{j=1}^n $  satisfies the  NSP of order $k\coloneqq \|c\|_0$. Then Theorem \ref{IFF} says that $c$ is the unique solution to Problem \ref{BP1}. Let  $x \in \mathcal{E}$ can be written as  $	x=\theta_\tau^* c$ for some  $c \in  \mathcal{A}^n$ satisfying $\|c\|_0\leq k$. Let $ M \subseteq   \{1, \dots, n\} $ with $o(M)\leq k$ and let $d=(d_j)_{j=1}^n \in \ker(\theta_\tau^*), d \neq 0$. Then we have 
\begin{align*}
	\theta_\tau\theta_\tau^* d=0.
\end{align*}
For each fixed $1\leq k \leq n$, above equation gives
\begin{align*}
	0&=\langle 	\theta_\tau\theta_\tau^*(d_j)_{j=1}^n, e_k \rangle=\langle 	\theta_\tau^*(d_j)_{j=1}^n, \theta_\tau^*e_k \rangle\\
	&=\langle 	\theta_\tau^*(d_j)_{j=1}^n, \tau_k \rangle
	=\sum_{j=1}^{n}d_j\langle \tau_j, \tau_k\rangle\\
	&=d_k\langle \tau_k, \tau_k\rangle+\sum_{j=1, j \neq k}^{n}d_j\langle \tau_j, \tau_k\rangle=d_k+\sum_{j=1, j \neq k}^{n}d_j\langle \tau_j, \tau_k\rangle.
\end{align*}
Therefore 
\begin{align*}
d_k=-\sum_{j=1, j\neq k}^nd_j\langle \tau_j, \tau_k\rangle, \quad \forall 1\leq k \leq n.	
\end{align*}
By taking norm,  
\begin{align*}
	\|d_k\|&=\left\|\sum_{j=1, j\neq k}^nd_j\langle \tau_j, \tau_k\rangle\right\|\leq \sum_{j=1, j\neq k}^n\|d_j\langle \tau_j, \tau_k\rangle\|\\
	&\leq \sum_{j=1, j\neq k}^n\|d_j\|\|\langle \tau_j, \tau_k\rangle\|
	\leq \left(\displaystyle\max_{1\leq j, k \leq n, j\neq k}\|\langle \tau_j, \tau_k\rangle\|\right)\sum_{j=1, j\neq k}^n\|d_j\|\\
	&=\left(\displaystyle\max_{1\leq j, k \leq n, j\neq k}\|\langle \tau_j, \tau_k\rangle\|\right)\left(\sum_{j=1}^n\|d_j\|-\|d_k\|\right)\\
	&=\left(\displaystyle\max_{1\leq j, k \leq n, j\neq k}\|\langle \tau_j, \tau_k\rangle\|\right)\left(\|d\|_1-\|d_k\|\right), \quad \quad \forall 1\leq k \leq n.
\end{align*}

By rewriting above inequality  we get
\begin{align}\label{S}
	\left(1+\frac{1}{\displaystyle\max_{1\leq j, k \leq n, j\neq k}\|\langle \tau_j, \tau_k\rangle\|}\right)\|d_k\|\leq \|d\|_1, \quad  \forall 1\leq k \leq n.
\end{align}
Summing Inequality (\ref{S}) over $M$ leads to 
\begin{align*}
		\left(1+\frac{1}{\displaystyle\max_{1\leq j, k \leq n, j\neq k}\|\langle \tau_j, \tau_k\rangle\|}\right)\|d_M\|_1&=	\left(1+\frac{1}{\displaystyle\max_{1\leq j, k \leq n, j\neq k}\|\langle \tau_j, \tau_k\rangle\|}\right)\sum_{k\in M}\|d_k\|\\
	&\leq \|d\|_1	\sum_{k \in M} 1= \|d\|_1o(M).
\end{align*}
Finally using Inequality (\ref{I})

\begin{align*}
\|d_M\|_1&\leq \left(1+\frac{1}{\displaystyle\max_{1\leq j, k \leq n, j\neq k}\|\langle \tau_j, \tau_k\rangle\|}\right)^{-1} \|d\|_1o(M)\\
&\leq\left(1+\frac{1}{\displaystyle\max_{1\leq j, k \leq n, j\neq k}\|\langle \tau_j, \tau_k\rangle\|}\right)^{-1} \|d\|_1k\\
&=\left(1+\frac{1}{\displaystyle\max_{1\leq j, k \leq n, j\neq k}\|\langle \tau_j, \tau_k\rangle\|}\right)^{-1} \|d\|_1\|c\|_0\\
&<\frac{1}{2}\|d\|_1.
\end{align*}
Hence 	$\{\tau_j\}_{j=1}^n $  satisfies the  NSP of order $k$.
\end{proof}
\begin{theorem} (\textbf{Noncommutative Donoho-Elad-Gribonval-Nielsen-Fuchs Sparsity Theorem})
Let  $\{\tau_j\}_{j=1}^n $  be a unit inner product frame for  $\mathcal{E}$.
If $x \in \mathcal{E}$ can be written as  $	x=\theta_\tau^* c$ for some  $c \in \mathcal{A}^n$ satisfying 
\begin{align*}
	\|c\|_0<\frac{1}{2}\left(1+\frac{1}{\displaystyle\max_{1\leq j, k \leq n, j\neq k}\|\langle \tau_j, \tau_k\rangle\|}\right),	
\end{align*}
then $c$ is the unique solution to Problem \ref{BP0}.	
\end{theorem}
\begin{proof}
Theorem \ref{TDP} says that  $c$ is the unique solution to Problem \ref{BP1}.  Let $d \in \mathcal{A}^n$ be such that 	$x=\theta_\tau^* d$. We claim that $\|d\|_0> \|c\|_0$.  If this fails, we must have $\|d\|_0\leq \|c\|_0$. We then have 
\begin{align*}
	\|d\|_0<\frac{1}{2}\left(1+\frac{1}{\displaystyle\max_{1\leq j, k \leq n, j\neq k}\|\langle \tau_j, \tau_k\rangle\|}\right).
\end{align*}
Theorem \ref{TDP} again says  that $d$ is also the unique solution to Problem \ref{BP1}. Therefore we must have  $\|c\|_1<\|d\|_1$ and $\|c\|_1>\|d\|_1$ which is a contradiction. So claim holds and we have $\|d\|_0> \|c\|_0$.
\end{proof}

 \bibliographystyle{plain}
 \bibliography{reference.bib}

\end{document}